\documentclass[12pt,a4paper]{article} 
\usepackage{amsmath}
\usepackage{amsfonts}
\usepackage{amssymb}
\usepackage{cite}
\usepackage{graphics}
\usepackage{amssymb,amscd}
\usepackage[utf8]{inputenc}
\usepackage[T2A]{fontenc}
\usepackage{tikz-cd}
\usepackage[russian,english]{babel}

\usepackage{amscd}
\usepackage{yfonts}
\usepackage[normalem]{ulem}
\usepackage{esint}
\usepackage{amsmath}
\usepackage{amsfonts}
\usepackage{amssymb}
\usepackage{hyperref}

\usepackage{graphics}
\usepackage{amssymb,amsmath,mathrsfs,amsthm,amscd}

\theoremstyle{plain}

\newtheorem{theo}{Theorem}[section]
\newtheorem{prop}[theo]{Proposition}

\newtheorem{lemm}[theo]{Lemma}

\theoremstyle{definition}

\theoremstyle{remark}
\newtheorem*{rema}{Remark}

\numberwithin{equation}{section}

\DeclareMathOperator{\tr}{tr}
\DeclareMathOperator{\re}{Re}
\DeclareMathOperator{\Res}{Res}

\DeclareMathOperator{\Vol}{vol}

\title{Spectral geometry of surfaces with curved conic singularities.}
\author{Asilya Suleymanova}
\date{November 2017}

\begin{document}
\maketitle

\begin{abstract}
Let $(M,g)$ be a surface with Riemannian metric and curved conic singularities. More precisely, a neighbourhood of a singularity is isometric to $(0,1)\times S^1$ with metric $g_{\text{conic}}=dr^2+f(r)^2d\theta^2, r\in(0,1)$. We study the spectral geometry of $(M,g)$ using the heat trace expansion. We express the first few terms in the expansion through the geometry of the singularities. The constant term contains information about the angle at the tip of the cone. The next term, $bt^{1/2}$, is expressed through the curvature and the angle at the tip of the cone.
\end{abstract}

\tableofcontents

\section{Introduction}

In the prequel paper \cite[Lemma~2.1, Theorem~2.2]{S2}, we studied the spectral geometry of surfaces with conic singularities, where $g_{\text{conic}}=dr^2+r^2d\theta^2, r\in(0,1)$. In this case the cone near the singularity is flat and there is one term in the heat trace expansion that contains information about the singularity. It is the constant term, which is expressed through the angle at the tip of the cone. In this article we study the spectral geometry in the more general case, when the cone near the singularity may be non-flat. As expected there is more then one term in the heat trace expansion which contains information about the singularity. The constant term, as in the flat case, contains information about the angle at the tip of the cone. The next term, $bt^{1/2}$, is expressed through the curvature of the cone and the angle.

Let $(M,g)$ be a non-complete smooth surface with Riemannian metric that possesses a curved conic singularity. By this we mean that there is an open set $U$ such that $M\setminus U$ is smooth compact surface with closed smooth boundary $(N,g_N)$. Furthermore, $U$ is isometric to $(0,1)\times N$ and the metric on $(0,1)\times N$ is

\begin{align}\label{conic metric}
g_{\text{conic}}=dr^2+f(r)^2d\theta^2, \;\;\; r\in(0,1).
\end{align}
Above $f(r)\in C^\infty([0,1))$, $f(0)=0$. From now on we refer to such singularity simply as conic singularity.

The metric on $(M,g)$ determines the Laplace-Beltrami operator $\Delta$. It is a nonnegative symmetric operator in $L^2(M)$ and hence admits self-adjoint extensions. Since $(M,g)$ is incomplete the Laplace-Beltrami operator may have many self-adjoint extensions, in this article we consider the Friedrichs extension. For a manifold with isolated conic singularities the discreteness of this extension is well known, see \cite[Section 3 and 4]{Ch}. It is also known that the heat operator $e^{-t\Delta}$ for $t>0$ is a trace class operator and there exists an asymptotic expansion as $t\to0+$, see \cite[Section 7]{BS2}. The main result of this article is the following theorem, where we give precise description of the terms in the expansion and compute the contribution of the singularity to the first terms.

\begin{theo}\label{heat trace on surface}
Let $(M,g)$ be a surface with $k$ conic singularities with the metric near the $i$-th singularity $g_{\text{conic}}=dr^2+f_i(r)^2d\theta^2, r\in(0,1)$. If $\Delta$ is the Friedrichs extension of the Laplace-Beltrami operator on $(M,g)$, then we have the asymptotic expansion
\begin{equation}
\begin{split}
\tr e^{-t\Delta}
\sim_{t\to0+}
\frac{1}{4\pi}\sum_{j=0}^\infty a_{j} t^{j-1}
+\sum_{j=0}^\infty b_{\frac{j}{2}} t^{\frac{j}{2}}
+\sum_{j=0}^\infty c_j t^j\log t,
\end{split}
\end{equation}

where $a_0=\Vol(M)$ and $c_0=0$. Furthermore
$$
b_0
=\frac{1}{12}\sum_{i=1}^{k}\left(\frac{1}{\sin\alpha_i}-\sin\alpha_i\right)
$$
and
$$
b_\frac{1}{2}
=\frac{5}{96\sqrt{\pi}}\sum_{i=1}^k\kappa_i\cot\alpha_i.
$$
Above $\alpha_i$ is the opening angle at the tip of the cone near the $i$-th the singularity and $\kappa_i$ is the curvature of the curve generating the cone near the $i$-th singularity, i.e. $\kappa_i=f_i''(0)$.
\end{theo}

Note that the surface $(M,g)$ does not have a boundary. In the case of the smooth compact surface without a boundary, half power order terms do not appear in the expansion. A very interesting phenomenon is that in the case of the surface with conic singularities in general there is a non-zero half power order coefficient $b_{\frac12}$.

\section{Regular Singular asymptotics}

In this section we briefly outline the properties of the regular singular operators in order to apply the Singular Asymptotic Lemma, \cite[p.~372]{BS2}, and obtain the resolvent trace expansion. We then use the results of this section to prove Theorem~\ref{heat trace on surface}.

Consider an operator
$$
L:=-dr^2+r^{-2}A(r), \;\;\;\;\;r>0,
$$
where $A(r)$ is a family of unbounded operators in a Hilbert space $H$ satisfying the conditions (A1)--(A6) in \cite[p.~373]{BS2}. Denote the Friedrichs extension of $L$ by the same letter. It is shown in \cite[pp. 400--409]{BS} that for any function $\varphi\in C^\infty_c(\mathbb{R})$ the operator $\varphi(r)(L+z^2)^{-d}$ is trace class for $d>1$. By \cite[Appendix: A Trace Lemma]{BS2},

\begin{lemm}
The operator $\varphi(r)(L+z^2)^{-d}$ has a kernel, we denote it by $\left[\tr_{L^2(N)}\left(\varphi(r)(L+z^2)^{-d}\right)\right](r_1,r_2)$such that 
$$
\varphi(r)(L+z^2)^{-d}f(r)=\int_{-\infty}^\infty\left[\tr_{L^2(N)}\left(\varphi(r)(L+z^2)^{-d}\right)\right](r,q)f(q)dq.
$$
Furthermore
$$
\tr\left(\varphi(r)(L+z^2)^{-d}\right)=\int_{-\infty}^\infty\left[\tr_{L^2(N)}\left(\varphi(r)(L+z^2)^{-d}\right)\right](q,q)dq.
$$
\end{lemm}

Let $L_r$ be a scaled operator as in \cite[(4.12)]{BS2}.

We denote $\sigma_r(r,\zeta):=\left[\tr_H(L_r+\zeta^2/r^2)^{-d}\right](r,r)$. By \cite[Lemma~4.9]{BS2}, we have the scaling property

\begin{align}\label{scaling property}
\sigma_1(r,\zeta)=r^{2d-1}\varphi(r)\left[\tr_H(L_r+\zeta^2)^{-d}\right](1,1).
\end{align}

We assume that 
\begin{equation}\label{sigma expansion series}
\sigma(r,\zeta):=\sigma_1(r,\zeta)\sim_{\zeta\to\infty}\sum_{l=0}^\infty\zeta^{-2d+2-l}\sigma_l(r).
\end{equation}


\begin{prop}\label{general resolvent expansion}
Let $\varphi(r)$ be a smooth function with compact support in $\mathbb{R}$ and $\varphi=1$ near $r=0$. We have the following expansion as $z\to\infty$
\begin{equation}\label{resolvent trace expansion}
\begin{split}
\tr\left(\varphi(r)(L+z^2)^{-d}\right)
\sim
\sum_{l=0}^\infty a^\rho_l z^{-2d+2-l}
+\sum_{l=0}^\infty b^\rho_l z^{-2d-l}
+\sum_{l=0}^\infty c^\rho_l z^{-2d-l}\log z,
\end{split}
\end{equation}
where the interior terms
\begin{align*}
a^\rho_l
=\fint_0^{\infty}\varphi(r)r^{-2d+2-l}\sigma_{l}(r)dr,
\end{align*}
the singular terms
\begin{align*}
b^\rho_l
=\frac{1}{(2d-1+l)!}\fint_0^{\infty}\zeta^{2d-1+l}\partial^{2d-1+l}_r\left(\sigma(r,\zeta)\right)|_{r=0}d\zeta,
\end{align*}
and the logarithmic terms
\begin{align*}
c^\rho_l
=\frac{1}{(2d-3+2l)!}\partial^{2d-3+2l}_r(\sigma_l(r))|_{r=0}.
\end{align*}
\end{prop}

\begin{rema}
The subscript $\rho$ indicates that these are coefficients in the resolvent expansion. The coefficients in the heat trace expansion we denote without the subscript.
\end{rema}

\begin{proof}
By \cite[Theorem~5.2]{BS2}, we can apply the Singular Asymptotics Lemma to the resolvent kernel $\sigma(r,\zeta)$ supported near $r=0$ to get the expansion of the trace of the resolvent near $r=0$. We use (\ref{sigma expansion series}) to obtain

\begin{equation}\label{preliminary expansion}
\begin{split}
\tr\left(\varphi(r)(L+z^2)^{-d}\right)
=&\fint_0^{\infty}\varphi(r)\sigma(r,\zeta)dr\\
\sim
&\sum_{l=0}^{\infty}\fint_0^{\infty}\sigma_{l}(r)(rz)^{-2d+2-2l}\varphi(r)dr\\
&+\sum_{l=0}^{\infty}z^{-l-1}\frac{1}{l!}\fint_0^{\infty}\zeta^l\partial^l_r\left(\sigma(r,\zeta)\right)|_{r=0}d\zeta\\
&+\sum_{l=0}^{\infty}z^{-2d-2l}\log z\frac{\partial^{2d-3+2l}_r\left(\sigma_{l}(r)\right)|_{r=0}}{(2d-3+2l)!}.
\end{split}
\end{equation}

By the scaling property (\ref{scaling property}), all derivatives in the second sum in (\ref{preliminary expansion}) for $l<2d-1$ are zero. Hence the first non zero summand in the sum appears for $l=2d-1$.
\end{proof}

\section{Computations of the singular terms}

Let $(M,g)$ be a surface with a conic singularity with the metric (\ref{conic metric}). The Laplace-Beltrami operator corresponding to (\ref{conic metric}) is 
$$
-f^{-1}(\partial_r f\partial_r)-f^{-2}\partial^2_\theta.
$$ 
We use the change of dependent variable $u\mapsto f^{1/2}u$ to see that the operator is unitary equivalent to 
$$
\Delta
:=-\partial_r^2-f^{-2}\partial^2_\theta+\frac12\frac{f''}{f}-\frac14\left(\frac{f'}{f}\right)^2.
$$

Let $h(r)\in C^\infty([0,1))$ be such that $f(r)=rh(r)$, $h(0)=f'(0)$ and $2h'(0)=f''(0)$. We obrain
\begin{align*}
\Delta
&=-\partial^2_r+r^{-2}\left(-h^{-2}(r)\partial^2_\theta+\frac12\frac{2rh'(r)+r^2h''(r)}{h(r)}-\frac14\frac{(h(r)+rh'(r))^2}{h^2(r)}\right)\\
&=:-\partial^2_r+r^{-2}A(r).
\end{align*}

To compute the derivative of $A(r)$, we simplify the notation and omit the argument $h=h(r)$
\begin{align*}
A'(r)
=&2h^{-3}h'\partial^2_\theta+\frac12\frac{(2h'+2rh''h'+2rh''+r^2h'''h')h+(2rh'+r^2h'')h'}{h^2}\\
&-\frac14\frac{2(h+rh')(2h'+rh''h')h^2+2hh'(h+rh')^2}{h^4}.
\end{align*}
 
Consequently
\begin{align*}
A(0)=-\frac{1}{(f'(0))^2}\partial^2_\theta-\frac14
\end{align*}
and 

\begin{align}\label{A'(0)}
A'(0)
=\frac{f''(0)}{(f'(0))^3}\partial^2_\theta-\frac14\frac{f''(0)}{f'(0)}.
\end{align}

The next proposition is similar to \cite[Proposition~4.14]{BL}.

\begin{prop}\label{one dimensional kernel}
Let $\nu\geq0, d>1$ and $T_\nu:=-\partial_r^2+r^{-2}\left(\nu^2-\frac14\right)$ be an operator in $L^2(\mathbb{R}_+)$. Denote its Friedrichs extension by the same letter. If $(T_\nu+z^2)^{-d-1}(r,r)$ is the kernel of the resolvent of $T_\nu$, then we have for $\text{max }\{-2-2d,-2-2\nu\}<\re(s)<-1$
$$
\fint_0^\infty r^s(T_\nu+1)^{-d-1}(r,r)dr
=\frac{\Gamma(d+1+\frac{s}{2})\Gamma(-\frac12-\frac{s}{2})\Gamma(\nu+1+\frac{s}{2})}{4\sqrt{\pi}d!\Gamma(\nu-\frac{s}{2})}.
$$
\end{prop}
\begin{proof}
Let $I_\nu(r)$ and $K_\nu(r)$ be the modified Bessel functions. By \cite[Proposition~4.2]{S1}, we obtain

\begin{align*}
\fint_0^\infty r^s(T_\nu+1)^{-d-1}(r,r)dr
=&\frac{1}{d!}\fint_0^\infty r^s\left(-\frac{1}{2\zeta}\frac{\partial}{\partial\zeta}\right)^drI_\nu(r\zeta)K_\nu(r\zeta)|_{\zeta=1}dr\\
=&\frac{1}{d!}\fint_0^\infty r^{s+2d+1}\left(-\frac{1}{2r}\frac{\partial}{\partial r}\right)^drI_\nu(r)K_\nu(r)dr.
\end{align*}
By \cite[p.123]{O},
\begin{align*}
\frac{1}{d!}\fint_0^\infty r^{s+2d+1}\left(-\frac{1}{2r}\frac{\partial}{\partial r}\right)^drI_\nu(r)K_\nu(r)dr\\
=\frac{\Gamma(d+1+\frac{s}{2})\Gamma(-\frac12-\frac{s}{2})\Gamma(\nu+1+\frac{s}{2})}{4\sqrt{\pi}d!\Gamma(\nu-\frac{s}{2})}.
\end{align*}
\end{proof}

Now we derive the expansion of the kernel of the resolvent on the diagonal away from the conic singularity. First we recall the classical theorem for the heat kernel expansion on the diagonal.

\begin{theo}[{\cite[Section III.E]{BGM}} ] \label{local_heat_expansion}
Let $K\subset M$ be any compact set and $p\in K$. There is an asymptotic expansion of the heat kernel along the diagonal
$$
\lVert e^{-t\Delta}(p)-(4\pi)^{-1}\sum_{i=0}^{j}t^{i-1}u_i(p)\rVert\leq C_j(K)t^{j+1},
$$
where $C_j(K)$ is some constant which depends on the compact set $K$. Moreover, $u_0(p)\equiv1$ and all $u_i(p)$ are polynomials on the curvature tensor and its covariant derivatives.
\end{theo}

Using the Cauchy's differentiation formula, we obtain the expansion of the kernel of the resolvent along the diagonal for $p\in K\subset M$

\begin{equation}
\begin{split}
\left\Vert(\Delta+z^2)^{-d}(p)-
(4\pi)^{-1}\sum_{l=0}^{k}z^{-2d+2-2l}u_l(p)\frac{\Gamma(d+l-1)}{(d-1)!}\right\Vert
\leq \tilde{C}_k(K)z^{-2d-2k}.
\end{split}
\end{equation}

Denote $\sigma(r,\zeta):=\tr_{L^2(S^1)}(\Delta+\zeta^2/r^2)^{-d}$, then

\begin{equation}
\sigma(r,\zeta)\sim_{\zeta\to\infty}\sum_{l=0}^\infty\zeta^{-2d+2-2l}\sigma_l(r),
\end{equation}
where
\begin{equation}
\sigma_l(r)=(4\pi)^{-1}\frac{\Gamma(d+l-1)}{(d-1)!}r^{2d-1+2l}\int_{S^1}u_l(r,\theta)d\theta.
\end{equation}

As a consequence of Proposition~\ref{general resolvent expansion}, we obtain

\begin{prop}\label{proposition resolvent trace expansion}
Let $\varphi(r)$ be a smooth function with compact support in $\mathbb{R}$ and $\varphi=1$ near $r=0$. We have the following expansion as $z\to\infty$
\begin{equation}
\begin{split}
\tr\left(\varphi(r)(\Delta+z^2)^{-d}\right)
\sim
\sum_{l=0}^\infty a^\rho_l z^{-2d+2-2l}
+\sum_{l=0}^\infty b^\rho_l z^{-2d-l}
+\sum_{l=0}^\infty c^\rho_l z^{-2d-2l}\log z,
\end{split}
\end{equation}
where the interior terms
\begin{align}\label{interior resolvent terms}
a^\rho_l
=(4\pi)^{-1}\frac{\Gamma(d+l-1)}{(d-1)!}\fint_0^{\infty}\varphi(r)r^{-2d+2-2l}\int_{S^1}u_l(r,\theta)d\theta dr,
\end{align}
the singular terms
\begin{align}\label{singular resolvent terms}
b^\rho_l
=\frac{1}{(2d-1+l)!}\fint_0^{\infty}\zeta^{2d-1+l}\partial^{2d-1+l}_r\left(\sigma(r,\zeta)\right)|_{r=0}d\zeta,
\end{align}
the logarithmic terms
\begin{align}\label{logarithmic resolvent terms}
c^\rho_l
=(4\pi)^{-1}\frac{\Gamma(d+l-1)}{(d-1)!(2d-3+2l)!}\partial^{2d-3+2l}_r\left(\int_{S^1}u_{l}(r,\theta)d\theta\right)|_{r=0}.
\end{align}
In particular, $a^\rho_0=\Vol(M)$, $c^\rho_0=0$, $c^\rho_1=\frac{d}{30}\text{Res}_1\frac{(f'')^2}{f}(0)$.
\end{prop}


We denote by $\alpha$ the angle between the axis of the cone and the line tangent to the surface of the cone at the tip of the cone, i.e. $\tan\alpha=f'(0)$. We denote by $\kappa:=f''(0)$ the curvature of the generating curve near the tip of the cone.

\begin{prop}\label{constant term}
We have
$$
b^\rho_0
=\frac{1}{12}\left(\frac{1}{\sin\alpha}-\sin\alpha\right).
$$
\end{prop}
\begin{proof}
See \cite[p.~424]{BS2}.
\end{proof}


\begin{prop}\label{b_1}
We have
$$
b^\rho_1
=\kappa\cot\alpha
\frac{5\Gamma(d+\frac{1}{2})}{96\sqrt{\pi}(d-1)!}.
$$
\end{prop}

\begin{proof}

Since
$$
\frac{d}{dr}(\Delta+1)^{-1}=-(\Delta+1)^{-1}r^{-1}A'(r)(\Delta+1)^{-1},
$$
we have
\begin{align}\label{derivative}
\tr_{L^2(\mathbb{R},L^2(S^1))}\left(\frac{d}{dr}(\Delta+1)^{-d}\right)
=-d\tr_{L^2(\mathbb{R},L^2(S^1))}\left(r^{-1}A'(0)(\Delta+1)^{-d-1}|_{r=0}\right).
\end{align}
By (\ref{singular resolvent terms}),

\begin{align*}
b^\rho_1
=&\frac{1}{(2d)!}\fint_0^{\infty}\zeta^{2d}\partial^{2d}_r\left(\tr_{L^2(\mathbb{R},L^2(S^1))}(\Delta+\zeta^2/r^2)^{-d}(r,r)\right)|_{r=0}d\zeta,
\end{align*}
by the scaling property (\ref{scaling property}) and (\ref{derivative}),

\begin{align*}
b^\rho_1
=&\partial_r\tr_{L^2(\mathbb{R},L^2(S^1))}(\Delta+1)^{-d}(r,r)|_{r=0}\\
=&-d\tr_{L^2(\mathbb{R},L^2(S^1))}\left(r^{-1}A'(0)(\Delta+1)^{-d-1}|_{r=0}\right).
\end{align*}

We use (\ref{A'(0)}) to compute further

\begin{align*}
b^\rho_1
=&-d\tr_{L^2(\mathbb{R},L^2(S^1))}\left(r^{-1}\left(\frac{f''(0)}{(f'(0))^3}\partial^2_\theta-\frac14\frac{f''(0)}{f'(0)}\right)(\Delta+1)^{-d-1}|_{r=0}\right)\\
=&-d\frac{f''(0)}{f'(0)}\Res_0|_{s=-1}\tr_{L^2(\mathbb{R},L^2(S^1))}\left(r^s\left(\frac{1}{(f'(0))^2}\partial^2_\theta-\frac14\right)(\Delta+1)^{-d-1}|_{r=0}\right)\\
=&-d\frac{f''(0)}{f'(0)}\Res_0|_{s=-1}F(s).
\end{align*}
Above $\Res_0|_{s=-1}F(s)$ denotes the regular analytic continuation at $s=-1$ of the function

\begin{align*}
F(s)
:=&\tr_{L^2(\mathbb{R},L^2(S^1))}\left(r^s\left(\frac{1}{(f'(0))^2}\partial^2_\theta-\frac14\right)(\Delta+1)^{-d-1}|_{r=0}\right)\\
=&2\sum_{k=1}^{\infty}\left(-\frac{1}{(f'(0))^2}k^2-\frac14\right)\tr_{L^2(\mathbb{R},L^2(S^1))}\left(r^s(T_k+1)^{-d-1}|_{r=0}\right)\\
&-\frac14\tr_{L^2(\mathbb{R},L^2(S^1))}\left(r^s(T_0+1)^{-d-1}|_{r=0}\right),
\end{align*}
where $T_\nu:=-\partial_r^2+r^{-2}\left(\nu^2-\frac14\right)$ is an operator in $L^2(\mathbb{R})$.

By Proposition~\ref{one dimensional kernel}, we have

\begin{align*}
F(s)
=&2\sum_{k=1}^{\infty}\left(-\frac{1}{(f'(0))^2}k^2-\frac14\right)\frac{\Gamma(d+1+\frac{s}{2})\Gamma(-\frac12-\frac{s}{2})\Gamma(k+1+\frac{s}{2})}{4\sqrt{\pi}d!\Gamma(k-\frac{s}{2})}\\
&-\frac14\frac{\Gamma(d+1+\frac{s}{2})\Gamma(-\frac12-\frac{s}{2})\Gamma(1+\frac{s}{2})}{4\sqrt{\pi}d!\Gamma(-\frac{s}{2})}\\
=&\Bigg(2\sum_{k=1}^{\infty}\left(-\frac{1}{(f'(0))^2}k^2-\frac14\right)\frac{\Gamma(k+1+\frac{s}{2})}{\Gamma(k-\frac{s}{2})}\\
&-\frac14\frac{\Gamma(1+\frac{s}{2})}{\Gamma(-\frac{s}{2})}\Bigg)
\frac{\Gamma(d+1+\frac{s}{2})\Gamma(-\frac12-\frac{s}{2})}{4\sqrt{\pi}d!},
\end{align*}
hence

\begin{align*}
b^\rho_1
=&-d\frac{f''(0)}{f'(0)}
\Bigg(2\sum_{k=1}^{\infty}\left(-\frac{1}{(f'(0))^2}k^2-\frac14\right)\frac{\Gamma(k+\frac{1}{2})}{\Gamma(k+\frac{1}{2})}
-\frac14\frac{\Gamma(\frac{1}{2})}{\Gamma(\frac{1}{2})}\Bigg)
\frac{\Gamma(d+\frac{1}{2})}{4\sqrt{\pi}d!}\\
=&-\frac{f''(0)}{f'(0)}
\Bigg(2\sum_{k=1}^{\infty}\left(-\frac{1}{(f'(0))^2}k^2-\frac14\right)
-\frac14\Bigg)
\frac{\Gamma(d+\frac{1}{2})}{4\sqrt{\pi}(d-1)!}\\
=&-\frac{f''(0)}{f'(0)}
\Bigg(2\left(-\frac{1}{(f'(0))^2}\zeta(-2)-\frac14\zeta(-1)\right)
-\frac14\Bigg)
\frac{\Gamma(d+\frac{1}{2})}{4\sqrt{\pi}(d-1)!}\\
=&-\frac{f''(0)}{f'(0)}
\Bigg(\frac{1}{24}
-\frac14\Bigg)
\frac{\Gamma(d+\frac{1}{2})}{4\sqrt{\pi}(d-1)!}\\
=&\frac{f''(0)}{f'(0)}
\frac{5\Gamma(d+\frac{1}{2})}{96\sqrt{\pi}(d-1)!}.
\end{align*}
Above $\zeta(s)$ is the Riemann zeta function.

Use $f''(0)=\kappa$ and $f'(0)=\tan\alpha$ to obtain the final formula
$$
b^\rho_1
=\kappa\cot\alpha
\frac{5\Gamma(d+\frac{1}{2})}{96\sqrt{\pi}(d-1)!}.
$$
\end{proof}

\begin{proof}[Proof of Theorem~\ref{heat trace on surface}]
To proof the theorem we choose a partition of unity $\varphi_i, \psi$ such that $\varphi_i$ is supported near the $i$-th singularity and $\psi$ is supported away from the singularities. Then we consider
$$
\tr(\Delta+z^2)^{-d}=\sum_{i=1}^k\tr\left(\varphi_i(r)(\Delta+z^2)^{-d}\right)+\tr\left(\psi (\Delta+z^2)^{-d}\right).
$$
By Proposition~\ref{resolvent trace expansion}, we compute the coefficients in the expansion near each singularity, $\tr\left(\varphi_i(r)(\Delta+z^2)^{-d}\right)$. The coefficients away from the singularity are computed using the interior resolvent kernel expansion. Then we combine Proposition~\ref{proposition resolvent trace expansion} with Proposition~\ref{constant term} and Proposition~\ref{b_1}. We use $a_l=\frac{(d-1)!}{\Gamma(d-1+l)}a_l^\rho$ and $b_{l/2}=\frac{(d-1)!}{\Gamma(d+l/2)}b_l^\rho$ and $c_l=\frac{(d-1)!}{\Gamma(d+l)}c_l^\rho$ to obtain the coefficients in the heat kernel expansion from the coefficients in the resolvent kernel expansion. This finishes the proof.
\end{proof}

\end{document}